\documentclass{amsart}

\usepackage{color}

\usepackage{mathtools}
\usepackage{braket}
\usepackage{amssymb}

\DeclarePairedDelimiter\abs{\lvert}{\rvert}%
\DeclarePairedDelimiter\norm{\lVert}{\rVert}%
\DeclarePairedDelimiter\angleBrac{\langle}{\rangle}%

\DeclareMathOperator{\Trace}{\mathrm{Tr}}
\DeclareMathOperator{\diag}{\mathrm{diag}}

\DeclareMathOperator{\Imag}{\mathrm{Im}}

\makeatletter
\let\oldabs\abs
\def\abs{\@ifstar{\oldabs}{\oldabs*}}
\let\oldnorm\norm
\def\norm{\@ifstar{\oldnorm}{\oldnorm*}}
\let\oldangleBrac\angleBrac
\def\angleBrac{\@ifstar{\oldangleBrac}{\oldangleBrac*}}
\makeatother

\newtheorem{theorem}{Theorem}[section]
\newtheorem{lemma}[theorem]{Lemma}
\newtheorem{corollary}{Corollary}[theorem]

\newtheorem{assumption}{Assumption}

\theoremstyle{definition}
\newtheorem{definition}[theorem]{Definition}

\theoremstyle{remark}
\newtheorem{remark}[theorem]{Remark}

\numberwithin{equation}{section}

\newcommand{\bdrsuppress}[1]{}

\begin{document}

\title[The Local Ledoit-P\'{e}ch\'{e} Law]{The Local Ledoit-P\'{e}ch\'{e} Law}


\author{Van Latimer}
\author{Benjamin D$.$ Robinson}
\address{}
\curraddr{}
\email{}
\thanks{}



\date{}

\dedicatory{}

\begin{abstract}
  Ledoit and P\'{e}ch\'{e}, in \cite{LP}, proved convergence of certain functions of a random covariance matrix's resolvent; we refer to this as the \emph{Ledoit-P\'{e}ch\'{e} law}. One important application of their result is shrinkage covariance estimation with respect to so-called Minimum Variance (MV) loss, discussed in the work of Ledoit and Wolf \cite{ANS}. We provide an essentially optimal rate of convergence and hypothesize it to be the smallest possible rate of excess MV loss within the shrinkage class.
\end{abstract}

\maketitle
\section{Introduction}    \label{sec:Intro}

Let $X$ be an $M\times N$ matrix; we assume that $\frac{M}{N}$ converges to a limit $\phi$ as both $M$ and $N$ tend to $\infty$ (although this may be relaxed). Let
\begin{equation}
  \Sigma = VDV^*, \quad V = \begin{pmatrix} \mid & & \mid \\ \mathbf{v}_1 & \dots & \mathbf{v}_M \\ \mid & & \mid \end{pmatrix}, \quad D = \diag(\tau_1, \dots, \tau_M)
\end{equation}
with $\tau_1 \geq \cdots \geq \tau_M$, be an $M\times M$ real symmetric or complex Hermitian positive definite matrix together with its eigendecomposition.

We will make the following assumption about the ``training-data'' matrix $X$. 

\begin{assumption}    \label{assump:main}
$\Sigma$ is diagonal and the $M\times N$ matrix $X$ has i.i.d$.$ columns $\sim \mathcal{N}(0,\Sigma)$.
\end{assumption}

As is common in Random Matrix Theory, the dimensions $N$ and $M$ of $X$ and $\Sigma$, and of most every other matrix we will study, are assumed to go to infinity. Thus practically every major quantity of interest is a sequence of quantities, and all properties that we desire to study are those which emerge in the large dimensional limit. We therefore always, except perhaps when special emphasis is needed, suppress the dependence of matrices and functions thereof on the dimensions $N$, $M$, etc. We consider the sample covariance matrix
\begin{equation}
  S = \Sigma^{1/2} XX^* \Sigma^{1/2}
\end{equation}
We let
\begin{equation}
  S = ULU^*, \quad U = \begin{pmatrix} \mid & & \mid \\ \mathbf{u}_1 & \dots & \mathbf{u}_M \\ \mid & & \mid \end{pmatrix}, \quad L = \diag(\lambda_1, \dots, \lambda_M)
\end{equation}
with $\lambda_1 \geq \cdots \geq \lambda_M$ be its spectral decomposition.

It is a problem of great theoretical and practical interest to understand how the properties of the sample covariance matrix relate to properties of the population covariance matrix $\Sigma$. \bdrsuppress{(citation here?)}

Random Matrix Theory has had great success in the last decade in getting very fine control of random matrices $H$ by way of their resolvent
\begin{equation}
  (H-zI)^{-1}
\end{equation}
This was first done by Mar\v{c}enko and Pastur (\cite{MP67}). Their approach was to show that the trace of the resolvent of a random matrix approximately satisfies some self-consistent equation and then to reason that trace of the resolvent, which is also the Stieltjes transform of the emperical eigenvalue measure, must be close to the true solution to the self-consistent equation. This has remained a popular and powerful technique.  

One finds that the resolvent $R_M := R_M(z):= (S-zI)^{-1}$ of $S$ can be written as 
\begin{equation}
  \sum_{i=1}^M \frac{1}{\lambda_i - z} \mathbf{u}_i\mathbf{u}_i^*
\end{equation}

Ledoit and P\'{e}ch\'{e}, in their paper \cite{LP}, consider functions of the form
\begin{equation}
  \Theta(z):= \sum_{i=1}^M \sum_{j=1}^M \frac{1}{\lambda_i -z} \mathbf{u}_i\mathbf{u}_i^* g(\tau_j)\mathbf{v}_j\mathbf{v}_j^*
\end{equation}
for some function $g:\mathbb{R}\to \mathbb{R}$ with finitely many discontinuities, which amounts to a weighting of the spectral decomposition of the resolvent; components of the resolvent in different eigendirections of the population covariance matrix are weighted according to the value of $g$ applied to the associated eigenvalue of the population covariance matrix. $\Theta$ may be simplified as follows:
\begin{equation}
  \Theta(z) = \Trace((S-zI)^{-1} g(\Sigma))
\end{equation}
Here we recall that for a function $g:\mathbb{R}\to \mathbb{R}$ and a diagonal matrix $D$, we define a matrix $g(D)$ by 
\begin{equation}
  [g(D)]_{ij} := g(D_{ij})\delta_{ij},
\end{equation}
and for a real symmetric or complex Hermitian matrix $A$ with spectral decomposition $A = U^*DU$, we define
\begin{equation}
  g(A) := U^*g(D)U
\end{equation}

The case $g \equiv 1$ is of course of interest. In this case we have
\begin{equation}    \label{eq:CaseOfgEq1}
  \Theta(z) = \Trace((S-zI)^{-1}) = \mathcal{S}\left(\mu\right), \quad \mu := \sum_{i=1}^M \delta_{\lambda_i}
\end{equation}
where for a measure $\sigma$ we define the Stieltjes transform $\mathcal{S}(\sigma): \mathbb{H} \to \mathbb{H}$ by
\begin{equation}
  \mathcal{S}(\sigma)(z) := \int \frac{d\sigma(x)}{x-z}
\end{equation}
and where by $\mathbb{H}$ we denote the complex upper half-plane
\begin{equation}
  \mathbb{H} := \set{z= E + i\eta \in \mathbb{C} \mid \eta > 0}
\end{equation}
We note that $E$ and $\eta$ are usually the way we will denote the real and imaginary parts of a complex argument to a Stieltjes transform.

The quantity \eqref{eq:CaseOfgEq1} has been deeply understood for a fairly general class of matrices, which we will detail soon.

We will also particularly interested in the case that $g$ is the identity, ie, $g(x) \equiv x$. In this case, $\Theta$ has another simplification:
\begin{equation}    \label{eq:STofEstMeas}
  \Theta = \mathcal{S}\left(\nu\right)
\end{equation}
where
\begin{equation}
  \nu := \sum_{i=1}^M \mathbf{u}_i^*\Sigma \mathbf{u}_i \delta_{\lambda_i}
\end{equation}

The reason this case of $g$ is of particular interest to us is that the quantities
\begin{equation}
  \mathbf{u}_i^*\Sigma\mathbf{u}_i
\end{equation}
are precisely the quantities which describe the Frobenius-norm optimal rotation equivariant shrinkage estimator for the population covariance matrix, as shown in \cite{LP}. Another result of the same paper was bounding $\nu$ close to a deterministic measure was one result of, although they did not provide rates of convergence. In this paper we improve their result by providing essentially optimal rate of convergence.

\begin{definition} [Shrinkage Estimators and Loss Function]
  Given a realization of the sample covariance matrix $S$, we define a (rotation equivariant) shrinkage estimator $\hat{\Sigma}$ for the population covariance matrix $\Sigma$ via
  \begin{equation}
    \hat{\Sigma} = U \hat{D} U^*
  \end{equation}
  for some diagonal matrix $\hat{D}$. That is, we estimate the $\Sigma$ from $S$ by keeping the eigenvectors and changing the eigenvalues, presumably ``shrinking'' them since $S$ has the tendency to ``spread out'' the eigenvalues of $\Sigma$.

  To measure the success of $\hat{\Sigma}$ we define the loss function
  \begin{equation}
    \mathcal{L}^{MV} (\hat{\Sigma}, \Sigma) := \frac{\Trace(\hat{\Sigma}^{-1} \Sigma \hat{\Sigma}^{-1})/N}{\left[ \Trace \left(\hat{\Sigma}^{-1}\right)/N \right]^2} - \frac{1}{\Trace\left(\Sigma^{-1}\right)/N}
  \end{equation}
\end{definition}
Here $MV$ stands for minimum variance, and ``$\mathcal{L}^{MV}$ represents the \emph{true} variance of the linear combination of the original variables that has the minimum \emph{estimated} variance.'' (See \cite{ANS} for the quote and for more discussion of the suitability of this loss function.)

\begin{lemma}    \label{lem:optimalShrinkForm}
  With respect to $\mathcal{L}^{MV}$, the optimal shrinkage estimator is given when $\hat{D}_{ii} = \beta \mathbf{u}_i^*\Sigma \mathbf{u}$, where $\beta \neq 0$ is a scaling constant that we will take to be 1. 
\end{lemma}
\begin{definition}
  We define
  \begin{equation}
    D^{\mathrm{or}} := \mathrm{diag}(\mathbf{u}_i^*\Sigma \mathbf{u}_i)_i, \quad \Sigma^{\mathrm{or}} := U D^{\mathrm{or}} U^*
  \end{equation}
  and we call $\Sigma^{\mathrm{or}}$ the \emph{shrinkage oracle}. 
\end{definition}
\begin{remark}
  As noted above, the same choice of $\hat{D}$ is optimal for a Frobenius norm loss function \cite{LP}.
\end{remark}

The optimal shrunken eigenvalues $\mathbf{u}_i^*\Sigma \mathbf{u}_i$ are experimentally unavailable to us, so for statistical purposes Lemma \ref{lem:optimalShrinkForm} is of limited use to us. However, just as was done for $\mu$, in \cite{ALLRM}, namely, bounding it optimally close to a determinstic limit, we will do for $\nu$. 

Given our random matrix with independent entries $X$ and our population covariance matrix $\Sigma$, we define the resolvent, or Green function introduced in \cite{ALLRM}:
\begin{equation}
  G(z) := \begin{pmatrix} -\Sigma^{-1} & X \\ X^* & -zI \end{pmatrix}
\end{equation}
The reader accustomed to Random Matrix Theory will note that this is not the usual defintion of the resolvent. It is however an important realization made in \cite{ALLRM} that the more familiar resolvent
\begin{equation}
  R_M := (S-zI)^{-1}
\end{equation}
can be neatly gotten from $G$, as well as the related resolvent
\begin{equation}
  R_N := (X^* \Sigma X-zI)^{-1}
\end{equation}
and that this single matrix containing both resolvents unlocks powerful tools for studying resolvent estimates developed in the context of Wigner matrices. In fact, $G$ may be decomposed in block form as
\begin{equation}
  \begin{pmatrix} z\Sigma^{1/2}R_M\Sigma^{1/2} & \cdot \\ \cdot & R_N \end{pmatrix}
\end{equation}
where the blocks labeled $\cdot$ are not of interest to us currently. The conjugation of $R_M$  by $\Sigma^{1/2}$ is not of great importance to the authors of \cite{ALLRM}, but it is very fortunate for us, the reason being that $\Theta$, our object of greatest interest, is precisely
\begin{equation}
  \Theta = \Trace(R_M\Sigma) = \Trace (\Sigma^{1/2}R_M\Sigma^{1/2})
\end{equation}
by the invariance of $\Trace$ under cyclical permutation. Let us make a few more definitions and then quote a result of \cite{ALLRM}.

We define the the population spectral measure, or \emph{PSM}, of $\Sigma$ by
\begin{equation}
  \pi = \frac{1}{M}\sum_{j=1}^M \delta_{\tau_i}
\end{equation}
This is of course just the probability measure which places equal weight at each of $\Sigma$'s eigenvalues, counted with multiplicity.

We also define the following notation of size for random variables, introduced in \cite{ErKnYa13}, which has proven very helpful for formulating results in RMT. 

\begin{definition}[Stochastic Domination]
  Given two sequences of random variables $X:= \{X_N\}_{N\in\mathbb{N}}$ and $Y:= \{Y_N\}_{N\in\mathbb{N}}$ (note that we again suppress that quantities of interest are sequence in $N$), we say that $Y$ \emph{stochastically dominates} $X$, or that $X \prec Y$, if for any (small) $\epsilon > 0$, (large) $D > 0$, and sufficiently large $N$, we have
  \begin{equation}
    P(X > N^\epsilon Y) < N^{-D}
  \end{equation}
\end{definition}

A little more notation: we define the important function $m:\mathbb{H}\to \mathbb{H}$ as the unique such value solving
\begin{equation}    \label{eq:mDef}
  \frac{1}{m} = -z + \phi \int \frac{x}{1 + mx} \mathrm{d} \pi(x)
\end{equation}
for $z \in \mathbb{H}$. 

We will list some things that we know about $m$.
\begin{itemize}
\item The equation \eqref{eq:mDef}, which we have taken as the definition of $m$, is the one which allows us to make the connection between the different definitions of $\delta$ in the papers \cite{LP} and \cite{ALLRM}.
\item $m$ is also the unique solution to $f(m) = z$, where
  \begin{equation}
    f(x) = -\frac{1}{x} + \phi\sum_{\tau_i\neq 0} \frac{1}{x + \tau_i^{-1}}.
  \end{equation}
  
  
\item $\lim_{\mathbb{H} \ni z \to E} m(z) := \check{m}(E)$ exists and is given by
  \begin{equation}
    \check{m}(E) = \pi\left(\mathcal{H}w(E) + iw(E)\right)
  \end{equation}
  where $w:=\frac{d\varrho}{d\lambda}$ is the Radon-Nikodym derivative of $\varrho$ with respect to $\lambda$, and where $\mathcal{H}$ is the Hilbert transform (see \cite{ANS}). We mention Hilbert transforms because \cite{ANS} presents it this way and explains how the presence of the Hilbert transform provides a theoretical explanation for the phenomenon of eigenvalue shrinkage, but we will not heavily use the Hilbert transform in our treatment; we will use a equation from \cite{ALLRM} which $m$ satisfies to get control of $m$'s real and imaginary parts directly.

\end{itemize}

At this time, let us also define the shrinkage function
\begin{equation}    \label{eq:deltaDef}
  \delta(x) := \frac{x}{[\pi cx w(x)]^2 + [1-c-\pi cx\mathcal{H}w(x)]^2}
\end{equation}
This function appeared first in a slightly different form in \cite{LP} and then in its stated form in \cite{ANS}; in both cases it is useful to us as an approximator to the values $\mathbf{u}_i^* \Sigma \mathbf{u}_i$ which describe the optimal shrinkage estimator. 
\begin{remark}
  We note that there is a small discrepancy between our definition of $m$ and the definition of $m$ in the context of \cite{LP} for $ M/N \gtrsim 1+\epsilon$. However, this discrepancy only amounts to how the limiting empirical spectral measure weights 0, and thus can be easily accounted for. 
\end{remark}

\begin{theorem}[Informal Statement of \cite{ALLRM}'s main result]    \label{thm:ALLRMInformal}
  Define the matrix
  \begin{equation}
    \Pi := \Pi(z) := \begin{pmatrix} -\Sigma(I + m(z)\Sigma)^{-1} & 0 \\ 0 & m(z)I \end{pmatrix}
  \end{equation}
  Then element by element, $G$ is very close to $\Pi$.
\end{theorem}

\begin{theorem}[Slightly more Formal Statement of part of Theorem \ref{thm:ALLRMInformal}]    \label{thm:bottomCornerTrace}
  If $\Sigma$'s population spectral measure $\pi$ satisfies some mild regularity constraints, then
  \begin{equation}    
    \frac{1}{N}\Trace(R_N(z)) - m(z) = O_\prec\left((N\eta)^{-1}\right)
  \end{equation}
\end{theorem}
This result is essentially optimal (up to the definitions in $\prec$) and cannot be gotten naively. The paper also provides essentially optimal bounds on individual resolvent elements; individual diagonal entries of $R_N$ are themselves close to $m$, but the difference is of an order $(N\eta)^{-1/2}$ in the bulk spectrum; this means that in averaging the diagonal elements to get the normalized trace, there is a fair bit of cancellation between different diagonal elements, as there is between independent random variables. The task of finding the ``parts'' of the random variables $(R_N)_{ii}$ which are independent to one another and thus provide this cancellation is the content of a ``Fluctuating Averaging Lemma'' in random matrix theory.

A corollary of this result is the ``Marchenko-Pastur law on small scales''
\begin{corollary}    \label{cor:bottomCornerMeasures}
  For any interval $I \subseteq \mathbb{R}$, we have
  \begin{equation}
    \mu(I) = \varrho(I)  + O_\prec\left(N^{-1}\right)
  \end{equation}
\end{corollary}
This corollary is important in that it captures the fact that statements about Stieltjes transforms of measures, which we have in great strength thanks to the techniques of \cite{ALLRM}, can be translated into statements about the measures themselves. This statement captures the fact that empirical eigenvalue distribution is given very accurately by a certain deterministic distribution, even on very fine scales: it says ``even in intervals which are predicted to contain only $N^\epsilon$ eigenvalues of $S$ according the to deterministic measure $\varrho$, we do have that the prediction is correct to leading order with very high probability.''


What we will first do in this note is adapt the proof of \ref{thm:bottomCornerTrace} to deal also with the quantity
\begin{equation}
  \frac{1}{M} \Trace\left(zR_M\Sigma + \Sigma(I + m(z)\Sigma)^{-1})\right)
\end{equation}
which leads us to our first main result:
\begin{theorem}    \label{thm:topCornerTrace}
  If $\Sigma$ satisfies the same regularity conditions as required for Theorem \ref{thm:bottomCornerTrace}, then we have
  \begin{equation}
    \frac{1}{M} \Trace\left(zR_M\Sigma + \Sigma(I + m(z)\Sigma)^{-1}\right) = O_\prec \left((N\eta)^{-1}\right)
  \end{equation}
  Using equation \eqref{eq:mDef}, we may rewrite the limit $M^{-1}\Trace\left(-\Sigma(I + m(z)\Sigma)^{-1}\right)$ as
  \begin{equation}
    -\phi^{-1} \left( \frac{1}{zm}+1\right)
  \end{equation}
\end{theorem}

Just as Marchenko-Pastur Law at small scales followed from Theorem \ref{thm:bottomCornerTrace}, so does our second main result. First we observe that \cite[Theorem~4]{LP} is equivalent to saying that the function $\delta$ from \eqref{eq:deltaDef} is the Radon-Nikodym derivative of the limiting measure for $\nu$ against the deformed Marchenko Pastur law.  Our second main result adds a rate of convergence to this limiting behavior, as follows:
\begin{corollary}    \label{cor:topCornerMeasures}
  We have for any interval $I \subseteq \mathbb{R}$ that
  \begin{equation}
    \mathrm{d}\nu(I) - \delta\mathrm{d}\varrho(I) = O_\prec\left(N^{-1}\right).
  \end{equation}
\end{corollary}


If $\tilde{\Sigma}$ is the shrinkage estimator $\sum_{i=1}^M \delta(\lambda_i)\mathbf{u}_i\mathbf{u}_i^*$, then the above implies an order of error between $\tilde{\Sigma}$ and $\Sigma^{\text{or}}$:
\begin{equation}
    \left|\mathrm{tr}\left(\Sigma^{\text{or}} - \tilde{\Sigma}\right)\right| \prec N^{-1}.
\end{equation}
Boundedness and continuity of $1/\delta$, together with the Portmanteau theorem, can then be used to show:
\begin{equation}
    \mathcal{L}^{\text{MV}}(\Sigma^{\text{or}}, \Sigma) - \mathcal{L}^{\text{MV}}(\tilde{\Sigma}, \Sigma)  = O_\prec( N^{-1}).
\end{equation}
Further, we hypothesize that no \emph{bona fide} shrinkage estimator can make this error asymptotically smaller, in which case this would be the smallest possible excess MV loss within the shrinkage class, as claimed in the abstract.

\section{Relation to Previous Works}
This paper is, firstly, a direct successor to the papers \cite{LP} and \cite{ANS} which advances a program established therein using recent advances in RMT. Another important connection is to the papers \cite{Bai2007} and \cite{VESD}, which discuss a measure which is related to ours: for a fixed unit vector $\mathbf{x}$, they study
\begin{equation}
  F_{Q_1,\mathbf{x}} := \sum_{i=1}^M \abs{\mathbf{u}_i^*\mathbf{x}}^2 \delta_{\lambda_i}
\end{equation}
As in our context, they prove that this measure is close to a deterministic limit, which \cite{VESD} calls $F_{1c, \mathbf{x}}$ (up to adjusting from their context to ours). In particular, they prove
\begin{equation}
  F_{Q_1, \mathbf{x}}(I) - F_{1c, \mathbf{x}}(I) \prec M^{-1/2}
\end{equation}
The earlier paper \cite{Bai2007} establishes the optimality of the factor $M^{-1/2}$ by remarkably establishing the joint asymptotic Gaussian distribution of any $k$ different analytic functions integrated against $F_{Q_1, \mathbf{x}}$ (one way that \cite{VESD} differs from or improves on \cite{Bai2007} is in the very high probability with which the error bounds hold).

We can recover our measure $\mu$ from $F_{Q_1, \mathbf{x}}$: indeed, if $\{\mathbf{v}_1, \dots, \mathbf{v}_M\}$ are the eigenvectors of $\Sigma$, then
\begin{equation}
  \mu = M^{-1} \sum_{j=1}^M \pi_j F_{Q_1, \mathbf{v}_j}
\end{equation}
Similiarly, the limiting deterministic measures satisfy
\begin{equation}
  \delta \mathrm{d}\varrho = M^{-1} \sum_{j=1}^M \pi_j F_{1c, \mathbf{v}_j}
\end{equation}
So our main result, with the error weakened from $O_\prec\left(N^{-1}\right)$ to $O_\prec \left(N^{-1/2}\right)$, is a consequence of the results of \cite{VESD}. \bdrsuppress{ Moreover, our context allows for weaker assumptions on $\Sigma$ and on the moments of $X$ that the context of \cite{VESD} (I am pretty sure, at least; we should check this more).}

The improvement by a factor of $M^{-1/2}$ that occurs after averaging is exactly reminiscent of the central limit theorem, which hints at a sort of independence between the meaures $F_{Q_1, \mathbf{v}_j}$ and $F_{Q_1, \mathbf{v}_{j'}}$. Also note that this improvement of $M^{-1/2}$ does not reflect an improvement of our work over theirs; the error bound of $O_\prec\left(N^{-1/2}\right)$ gotten in \cite{VESD} is optimal, and it is only after the averaging over many different deterministic vectors $\mathbf{x}$ that one sees the the improvement. Thus, our work is to their work as an averaged local law is to an entry-wise local law.

Lastly, one should note that an ultimate goal of the program investigated in \cite{VESD} is to establish the convergence of the CDF of $F_{Q_1, \mathbf{x}}$ to a Brownian bridge, which would amount to finding some ``internal'' independence inside $F_{Q_1, \mathbf{x}}$ in the form of independence between the quantities $\angleBrac{\mathbf{u}_i, \mathbf{x}}$. The ``external'' indepdence that we have hinted at between the measures $F_{Q_1, \mathbf{v}_j}$ and $F_{Q_1, \mathbf{v}_{j'}}$ is not unrelated to this. 

The connections between the papers \cite{Bai2007}, \cite{VESD} and ours are perhaps deeper than we have realized; future work will hopefully bring further connections to light.

\section{Tools}
First, we extend the definition of matrix multiplication to matrices indexed by arbitrary sets. 
\begin{definition}
  Let $\mathcal{I}_i$ be a finite set for $i \in \{1, 2, 3, 4\}$. Let $A$ be a $\mathcal{I}_1 \times \mathcal{I}_2$ matrix, and let $B$ be a $\mathcal{I}_3 \times \mathcal{I}_4$ matrix. We define the matrix product $AB$ to be a $\mathcal{I}_1 \times \mathcal{I}_4$ matrix satisfying
  \begin{equation}
    (AB)_{ij} = \sum_{k \in \mathcal{I}_2 \cap \mathcal{I}_3} A_{ik}B_{kj}
  \end{equation}
\end{definition}

\begin{definition}
  Let $A$ be an invertible matrix $\mathcal{J} \times \mathcal{J}$ matrix and let $S$ be its inverse. We define the minor $S^{(i)}$ via
  \begin{equation}
    S^{(i)} = \left( (A_{jk} : j,k \in \mathcal{J} \setminus \{i\})\right)^{-1}
  \end{equation}
\end{definition}

\begin{lemma}[Resolvent Identities]    \label{lem:resolventIdentities}
  Let $A$ be an invertible $\mathcal{J} \times \mathcal{J}$ matrix and $S = A^{-1}$.
  \begin{enumerate}
  \item If $j,k \in \mathcal{J} \setminus \{i\}$, 
    \begin{equation}
      S^{(i)}_{jk} = S_{jk} - \frac{S_{ji} S_{ik}}{S_{ii}}
    \end{equation}
  \item If $i \neq j$, 
    \begin{equation}
      S_{ij} = -S_{ii} \left(AS^{(i)}\right)_{ij}
    \end{equation}

  \item We have
    \begin{equation}
      \frac{1}{S_{ii}} = S_{ii} - (A^*S^{(i)}A)_{ii}
    \end{equation}

  \end{enumerate}
\end{lemma}

\begin{proof}
  \begin{enumerate}
  \item By the definition of an the inverse, it suffices to show that
    \begin{equation}
      \sum_{l\in \mathcal{J} \setminus\{i\}}\left(S_{jl} - \frac{S_{ji} S_{il}}{S_{ii}}\right) A_{lk} = \delta_{jk}
    \end{equation}
    But the left-hand side indeed yields
    \begin{equation}
      \sum_{l\in \mathcal{J} \setminus\{i\}} S_{jl}A_{lk} - \frac{S_{ji}}{S_{ii}} \sum_{l\in \mathcal{J} \setminus\{i\}} S_{il}A_{lk} = (\delta_{jk} -S_{ji}A_{ik}) - \frac{S_{ji}}{S_{ii}} (\delta_{ik} - S_{ii} A_{ik}) = \delta_{jk}
    \end{equation}
    since we have assumed $k \neq i$.
    
  \item This proof is taken from \cite{SSERG2}. Using part $(1)$, we get
    \begin{equation}
      \begin{split}
        \left(AS^{(i)}\right)_{ij} = \sum_{k\in \mathcal{J} \setminus \{i\}} A_{ik} S^{(i)}_{kj} = \sum_{k\in \mathcal{J} \setminus \{i\}} A_{ik} \left(S_{kj} - \frac{S_{ki} S_{ij}}{S_{ii}}\right) \\
        = - A_{ii}S_{ij} - \frac{S_{ij}}{S_{ii}}(1-A_{ii}S_{ii}) = - \frac{S_{ij}}{S_{ii}}
      \end{split}
    \end{equation}
    as desired. 

  \item This is an immediate consequence of Schur's complement formula, wherein if
    \begin{equation}
      M = \begin{pmatrix} B & C \\ D & E \end{pmatrix}
    \end{equation}
    then
    \begin{equation}
      M^{-1} = \begin{pmatrix} (B - CE^{-1}D)^{-1} & * \\ * & * \end{pmatrix}
    \end{equation}
    provided all the inverses exist. To see how Schur's complement formula applies, it is helpful to write
    \begin{equation}
      (A^*S^{(i)}A)_{ii} = (e_i^*A^*)S^{(i)}(Ae_i)
    \end{equation}
    It is of course crucial that we use the correct definition of matrix multiplication here. 

  \end{enumerate}
\end{proof}

Next we apply these general matrix algebra facts to our specific resolvent $G$. 
\begin{corollary} [Resolvent Identities for $G$]    \label{cor:ResIds}
  \begin{enumerate}
  \item For $s, t \neq r$ we have
    \begin{equation}
      G^{(r)}_{st} = G_{st} - \frac{G_{sr} G_{rt}}{G_{rr}}
    \end{equation}
  \item If $\mu\neq \nu \in \mathcal{I}_N$, then
    \begin{equation}
      G_{\mu\nu} = -G_{\mu\mu} (X^*G^{(\mu)})_{\mu \nu} = -G_{\nu\nu} (G^{(\nu)}X)_{\mu \nu}
    \end{equation}
    Similarly if $i \neq j \in \mathcal{I}_M$, then
    \begin{equation}
      G_{ij} = -G_{ii}(XG^{(i)})_{ij} = -G_{jj} (G^{(j)}X^*)_{ij}
    \end{equation}
    Lastly if $i \in \mathcal{I}_M$ and $\mu \in \mathcal{I}_N$, then
    \begin{equation}
      G_{i\mu} = -G_{\mu\mu} (G^{(\mu)}X)_{i\mu}, \quad G_{\mu i} -G_{\mu\mu} (X^*G^{(\mu)})_{\mu i}
    \end{equation}

  \item We have
    \begin{equation}
      \frac{1}{G_{\mu\mu}} = -z - (X^*G^{(\mu)}X)_{\mu \mu}
    \end{equation}
    
  \end{enumerate}
\end{corollary}
\begin{proof}
  These all follow from the lemma; only the correct definition of matrix multiplication must be used, and one should note that many of the lemmas' conclusions are insensitive to the diagonal entries of $A$, so that one may see a
  $$ \begin{pmatrix} -\Sigma^{-1} & X \\ X^* & 0 \end{pmatrix} $$
  when one expects to see a
  $$ \begin{pmatrix} -\Sigma^{-1} & X \\ X^* & -zI \end{pmatrix} $$     
\end{proof}

Let us cite one of the main results of \cite{ALLRM}
\begin{theorem}[Entrywise Local Law]    \label{thm:EWLocalLaw}
  For any deterministic unit vectors $\mathbf{v},\mathbf{w}\in \mathbb{R}^{\mathcal{I}}$, one has
  \begin{equation}
    \angleBrac{\mathbf{v}, (G - \Pi)\mathbf{w}} = O_\prec (\Psi)
  \end{equation}
  where
  \begin{equation}
    \Psi:= \Psi(z) = \sqrt{\frac{\Imag  m(z)}{N\eta}} + \frac{1}{N\eta}
  \end{equation}
\end{theorem}

\section{Proof of Theorem \ref{thm:topCornerTrace}}    \label{sec:topCornerTrace}

First, we prove a partial result. We do not include many details, and only show how section 5 of \cite{ALLRM} may be quickly adapted to our setting. 
\begin{lemma}    \label{lem:topCornerTraceDiag}
  If $\Sigma$ is diagonal, then Theorem \ref{thm:topCornerTrace} holds. 
\end{lemma}

\begin{proof}
  This proof amounts to adapting section 5 of \cite{ALLRM} to address the top-left corner of thei matrix $G$. We use the notation of \cite{ALLRM} without further comment.

  Equation (5.15) of \cite{ALLRM} reads
  \begin{equation}
    1(\Xi) G_{ii} = 1(\Xi)\frac{-\sigma_i}{1 + m_N \sigma_i + \sigma_i Z_i + O_\prec\left(\sigma_i \Psi_\Theta^2\right)}
  \end{equation}
  A Taylor expansion on this quantity, justified because $Z_i = O_\prec \left( \Psi_\Theta\right)$ by Lemma 5.2 of \cite{ALLRM}, yields
  \begin{equation}
    1(\Xi) G_{ii} = 1(\Xi)\left[\frac{-\sigma_i}{1 + m_N \sigma_i } + \frac{\sigma_i}{(1 + m_N \sigma_i)^2}Z_i + O_\prec\left( \sigma_i \Psi_\Theta^2\right) \right]
  \end{equation}
  Averaging now over $i$ yields
  \begin{equation}    \label{eq:afterAveraging}
    1(\Xi)\frac{1}{M} \sum_i G_{ii} = 1(\Xi) \frac{1}{M} \sum_{i} \frac{-\sigma_i}{1 + m_N \sigma_i} + [Z]_M + O_\prec \left(\Psi_\Theta^2\right)
  \end{equation}
  Using that $(1 - \mathbb{E}_i) \frac{1}{G_{ii}} = Z_i$, we have by lemma 5.6 of \cite{ALLRM} that
  \begin{equation}    \label{eq:ZsubMBound}
    [Z]_M = O_\prec\left(\Psi_\Theta^2\right)
  \end{equation}
  Using equation (3.10) of \cite{ALLRM} to bound $\Theta$, we have $\Psi_\Theta \prec \frac{1}{\sqrt{N\eta}}$, so that our equation \eqref{eq:ZsubMBound} becomes
  \begin{equation}
    [Z]_M = O_\prec\left((N\eta)^{-1}\right)
  \end{equation}
  Equation (3.1) of \cite{ALLRM} also yields that $\Xi$ holds with high probability, so that equation \eqref{eq:afterAveraging} yields
  \begin{equation}
    \frac{1}{M} \sum_i G_{ii} = \frac{1}{M} \sum_{i} \frac{-\sigma_i}{1 + m_N \sigma_i} + O_\prec\left((N\eta)^{-1}\right)
  \end{equation}
  Furthermore, it is a result of section 5 of \cite{ALLRM} that
  \begin{equation}
    \abs{m_N - m} = O_\prec\left( (N\eta)^{-1}\right)
  \end{equation}
  so that $m_N$ may be replaced with $m$ (it is noted that $\abs{1 + m(z) \sigma_i} \geq \tau$ under our regularity assumptions). Finally, $\frac{1}{M} \sum_{i} \frac{-\sigma_i}{1 + m_N \sigma_i}$ is precisely $\frac{1}{M} \sum_i \Pi_{ii}$, and
  $$\sum_i G_{ii} = \Trace z\Sigma^{1/2} R_M\Sigma^{1/2} = z\Trace R_M \Sigma$$
  so that we are done.

\end{proof}

\begin{remark}
  In the recent paper \cite{VESD}, a very similar result to our \ref{lem:topCornerTraceDiag} is proven in their equation (3.13), under much more general moment assumptions on $X$; however, the resolvent used in that paper differs slightly from the one used in this paper. We hope to use some of their techniques perhaps to weaken some of the moment assumptions in our work. 
\end{remark}

\section*{Acknowledgements}
This work was supported by the US Air Force Office of Scientific Research, Lab Task number 19RYCOR036.  The views and opinions of this paper do not necessarily reflect the official positions of the Air Force.  The public affairs approval number of this document is AFRL-2021-2586.

\bibliographystyle{amsplain}

\end{document}